\RequirePackage{fix-cm}
\RequirePackage{fixltx2e}
\documentclass[oneside,english]{amsart}

\usepackage[latin9]{inputenc}
\usepackage{babel}
\usepackage{url}
\usepackage{amsthm}
\usepackage{amssymb}
\usepackage{esint}
\usepackage[unicode=true,pdfusetitle,
 bookmarks=true,bookmarksnumbered=false,bookmarksopen=false,
 breaklinks=false,pdfborder={0 0 1},backref=false,colorlinks=false]
 {hyperref}
 \makeatletter

\renewcommand{\theequation}{\theequation. \arabic{equation}}
\numberwithin{equation}{section}
\newtheorem{thm}{Theorem}[section]

\newtheorem{prop}{Proposition}[section]

\usepackage[numbers,sort&compress]{natbib}

\def\squarebox#1{\hbox to #1{\hfill\vbox to #1{\vfill}}}
\def\qed{\hspace*{\fill}
         \vbox{\hrule\hbox{\vrule\squarebox{.667em}\vrule}\hrule}\smallskip}
\begin{document}\large
\title[On An Appell series over finite fields]
{\Large An Appell series over finite fields}
\author{\small  Bing He}
\address{\small
College of Science, Northwest A\&F University,
   Yangling 712100, Shaanxi, People's Republic of China}
\email{yuhe001@foxmail.com; yuhelingyun@foxmail.com}
\thanks{The first author is the corresponding author.}

\author{\small Long Li}
\address{\small Department of Mathematics, East China Normal University, 500 Dongchuan Road, Shanghai
200241, People's Republic of China}
\email{lilong6820@126.com}

\author{\small  Ruiming Zhang}
\address{\small
College of Science, Northwest A\&F University,
   Yangling 712100, Shaanxi, People's Republic of China}
\email{ruimingzhang@outlook.com}



\begin{abstract}
\small In this paper we present a finite field analogue for  one of the Appell series. We shall derive
its transformations, reduction formulas as well as  generating functions.
\end{abstract}
\maketitle
\section{Introduction}

As in  \cite{BEW} and \cite {IR}  we let $\mathbb{F}_{q}$ denote the finite field of $q$ elements and $\widehat{\mathbb{F}^{*}_{q}}$ the group of multiplicative characters of $\mathbb{F}^{*}_{q}$ where $q$ is a power of a prime. We also let  $\chi$ be a character of $\mathbb{F}^{*}_{q}$  lifted to  $\mathbb{F}_{q}$ by setting $\chi(0)=0$,   $\overline{\chi}$ and $\varepsilon$ denote  the inverse of $\chi$ and the trivial character respectively. In the work we shall give a finite field analogue for the second Appell series $F_{2}$, derive its transformations, reduction formulas and generating functions.

The generalized hypergeometric function is defined by \cite{B}
$$ {}_{n+1}F_n \left(\begin{matrix}
a_0, a_1, \ldots , a_{n} \\
b_1, \ldots , b_n \end{matrix}
\bigg| x \right):=\sum_{k=0}^{\infty}\frac{(a_{0})_{k}(a_{1})_{k}\cdots(a_{n})_{k}}{k!(b_{1})_{k}\cdots(b_{n})_{k}}x^{k}$$
where $(z)_{k}$ is the Pochhammer symbol given by
\begin{equation*}
(z)_{0}=1,~(z)_{k}=z(z+1)\cdots(z+k-1)\text{ for } k\geq 1.
\end{equation*}

Greene in \cite{Gr}  developed a theory of hypergeometric functions over finite fields and proved many  transformation and summation identities for his  hypergeometric functions. There Greene introduced the notation
\begin{equation*}
{}_{2}F_1 \left(\begin{matrix}
A, B \\
C \end{matrix}
\bigg| x \right)^{G}=\varepsilon(x)\frac{BC(-1)}{q}\sum_{y}B(y)\overline{B}C(1-y)\overline{A}(1-xy)
\end{equation*}
for $A,B,C\in \widehat{\mathbb{F}_{q}}$ and $x\in \mathbb{F}_{q}. $  This definition  is
clearly a finite field analogue for the integral \cite{B}:
\begin{equation*}
{}_{2}F_1 \left(\begin{matrix}
a, b \\
c \end{matrix}
\bigg| x \right)=\frac{\Gamma(c)}{\Gamma(b)\Gamma(c-b)}\int_{0}^{1}t^b(1-t)^{c-b}(1-tx)^{-a}\frac{dt}{t(1-t)}.
\end{equation*}
He also defined a finite field analogue for the binomial coefficient as
\begin{equation*}
  {A\choose B}^{G}=\frac{B(-1)}{q}J(A,\overline{B}),
\end{equation*}
where the Jacobi sum $J(\chi,\lambda)$ is given by $$J(\chi,\lambda)=\sum_{u}\chi(u)\lambda(1-u).$$ For more information about the finite field analogue for the generalized hypergeometric functions, please see \cite{FL, M, EG}.

In this paper, for the sake of simplicity, we use the notation
\begin{equation*}
  {A\choose B}=q{A\choose B}^{G}=B(-1)J(A,\overline{B}).
\end{equation*}
Furthermore, we define the finite field analogue for the classic Gauss hypergeometric series as
\begin{equation*}
{}_{2}F_1 \left(\begin{matrix}
A, B \\
C \end{matrix}
\bigg| x \right)=q\cdot{}_{2}F_1 \left(\begin{matrix}
A, B \\
C \end{matrix}
\bigg| x \right)^{G}=\varepsilon(x)BC(-1)\sum_{y}B(y)\overline{B}C(1-y)\overline{A}(1-xy).
\end{equation*}
For any $A,B,C\in \widehat{\mathbb{F}_{q}}$ and $x\in \mathbb{F}_{q}$, by \cite [Theorem 3.6]{Gr}, then
\begin{equation}\label{e1-1}
  {}_{2}F_1 \left(\begin{matrix}
A, B \\
C \end{matrix}
\bigg| x \right)=\frac{1}{q-1}\sum_{\chi}{A\chi\choose \chi}{B\chi\choose C\chi}\chi(x)
\end{equation}
Similarly, the finite field analogue of the generalized hypergeometric series for any $A_{0},A_{1},\cdots, A_{n},\\
B_{1},\cdots,B_{n}\in \widehat{\mathbb{F}_{q}}$ and $x\in \mathbb{F}_{q}$ is defined by
\begin{equation*}
  {}_{n+1}F_n \left(\begin{matrix}
A_{0}, A_{1},\cdots, A_{n} \\
 B_{1},\cdots, B_{n} \end{matrix}
\bigg| x \right)=\frac{1}{q-1}\sum_{\chi}{A_{0}\chi\choose \chi}{A_{1}\chi\choose B_{1}\chi}\cdots{A_{n}\chi\choose B_{n}\chi}\chi(x).
\end{equation*}
The finite field analogue for the binomial theorem can be stated in the form:

\begin{thm}\emph{(Binomial theorem, see \cite [(2.5)]{Gr})} For any character $A\in \widehat{\mathbb{F}_{q}}$ and $x\in \mathbb{F}_{q},$ we have
\begin{equation*}
  A(1+x)=\delta(x)+\frac{1}{q-1}\sum_{\chi}{A\choose \chi}\chi(x),
\end{equation*}
where the summation is over all multiplicative characters of $\mathbb{F}_{q}$ and $\delta(x)$ is a function on $\mathbb{F}_{q}$ given by
\begin{equation*}
\delta(x)=\left\{
            \begin{array}{ll}
              1 & \hbox{if $x=0$} \\
              0 & \hbox{if $x\neq 0$}
            \end{array}.
          \right.
\end{equation*}
\end{thm}

In our notations, one of Greene's theorem is  much simpler:
\begin{thm}\emph{(See \cite [Theorem 4.9]{Gr})}For any characters $A,B,C\in \widehat{\mathbb{F}_{q}},$  we have
\begin{equation}\label{e1-2}
{}_{2}F_1 \left(\begin{matrix}
A, B \\
 C \end{matrix}
\bigg| 1 \right)=A(-1){B\choose \overline{A}C}.
\end{equation}
\end{thm}

There are many interesting double hypergeometric functions. Among them Appell's four functions may be the most important ones \cite{A, B, CA, S}:
\begin{align*}
F_{1}(a;b,b';c;x,y)&=\sum_{m,n\geq 0}\frac{(a)_{m+n}(b)_{m}(b')_{n}}{m!n!(c)_{m+n}}x^my^n,~|x|<1,~|y|<1,\\
F_{2}(a;b,b';c,c';x,y)&=\sum_{m,n\geq 0}\frac{(a)_{m+n}(b)_{m}(b')_{n}}{m!n!(c)_{m}(c')_{n}}x^my^n,~|x|+|y|<1,\\
F_{3}(a,a';b,b';c;x,y)&=\sum_{m,n\geq 0}\frac{(a)_{m}(a')_{n}(b)_{m}(b')_{n}}{m!n!(c)_{m+n}}x^my^n,~|x|<1,~|y|<1,\\
F_{4}(a;b;c,c';x,y)&=\sum_{m,n\geq 0}\frac{(a)_{m+n}(b)_{m+n}}{m!n!(c)_{m}(c')_{n}}x^my^n,~|x|^{\frac{1}{2}}+|y|^{\frac{1}{2}}<1.
\end{align*}

Inspired by Greene's work, the second author \emph{et al} in \cite{LLM} gave a finite field analogue for the Appell series $F_{1}$ and
proved some transformations, reduction formulas and  generating functions for their  function. Their finite field analogue for the Appell series $F_{1}$  is
\begin{equation*}
F_{1}(A;B,B';C;x,y)=\varepsilon(xy)AC(-1)\sum_{u}A(u)\overline{A}C(1-u)\overline{B}(1-ux)\overline{B'}(1-uy).
\end{equation*}
$F_{1}$ was chosen owing to the fact that it has an  integral representation
 \cite [Chapter IX]{B}:
\begin{equation*}
F_{1}(a;b,b';c;x,y)=\frac{\Gamma(c)}{\Gamma(a)\Gamma(c-a)}\int_{0}^{1}u^{a-1}(1-u)^{c-a-1}(1-ux)^{-b}(1-uy)^{-b'}du
\end{equation*}
where $0< \Re(a)<\Re(c).$

Motivated by the ideas in \cite{Gr},  \cite{LLM}, and the integral representation,
 \cite [Chapter IX]{B}
\begin{align*}
	F_{2}(a;b,b';c,c';x,y)&=\frac{\Gamma(c)\Gamma(c')}{\Gamma(b)\Gamma(b')\Gamma(c-b)\Gamma(c'-b')}\\
	&\cdot \int_{0}^1 \int_{0}^1u^{b-1}v^{b'-1}(1-u)^{c-b-1}(1-v)^{c'-b'-1}(1-ux-vy)^{-a}dudv,
\end{align*}
in the current work we define a finite field analogue for  the Appell series  $F_{2}$ as
\begin{align*}
  F_{2}(A;&B,B';C,C';x,y)\\
&=\varepsilon(xy)BB'CC'(-1) \sum_{u,v}B(u)B'(v)\overline{B}C(1-u)\overline{B'}C'(1-v)\overline{A}(1-ux-vy),
\end{align*}
where $A,B,B',C,C'\in \widehat{\mathbb{F}_{q}},~x,y\in \mathbb{F}_{q}$ and  both summations are over all the elements of $\mathbb{F}_{q}.$ For the simplicity, the factor $\frac{\Gamma(c)\Gamma(c')}{\Gamma(b)\Gamma(b')\Gamma(c-b)\Gamma(c'-b')}$ is excluded in the definition, and the factor  $\varepsilon(xy)\cdot BB'CC'(-1)$ is here for a better expression in terms of binomial coefficients.

The following theorem gives an alternative representation for $F_{2}(A;B,B';C,C';x,y).$
\begin{thm}\label{t1}For any $A,B,B',C,C'\in \widehat{\mathbb{F}_{q}} $ and $x,y\in \mathbb{F}_{q},$ we have
\begin{align*}
  F_{2}(A;B,B';C,C';x,y)&=\frac{1}{(q-1)^2}\sum_{\chi, \lambda}{A\chi\choose \chi}{A\chi\lambda\choose \lambda}{B\chi\choose C\chi}{B'\lambda\choose C'\lambda}\chi(x)\lambda(y)\\
&~+\overline{A}(-x)\overline{C'}(y)\overline{B'}C'(1-y){\overline{A}B\choose B\overline{C}},
\end{align*}
where both summations are over all multiplicative characters of $\mathbb{F}_{q}.$
\end{thm}
From the definition of $F_{2}(A;B,B';C,C';x,y),$  Theorem \ref{t1}, \eqref{e1-2} and Proposition \ref{pp2-1} (in Section 2),  we can easily deduce the following results.
\begin{prop}For any $A,B,B',C,C'\in \widehat{\mathbb{F}_{q}} $ and $x,y\in \mathbb{F}_{q},$ we have
\begin{align}
  F_{2}(A;B,B';C,C';x,y)&=F_{2}(A;B',B;C',C;y,x),\label{p11-1}\\
F_{2}(A;B,B';C,C';x,1)&=B'C'(-1){}_{3}F_2 \left(\begin{matrix}
A,B,A\overline{C'}\\
C,AB'\overline{C'} \end{matrix}
\bigg| x \right),\label{p11-2}\\
F_{2}(A;B,B';C,C';1,y)&=BC(-1){}_{3}F_2 \left(\begin{matrix}
A,B',A\overline{C}\\
C',AB\overline{C} \end{matrix}
\bigg| y \right). \notag
\end{align}
\end{prop}

Unlike $F_{1}$, the Appell series $F_{2}$ does not  have  a single integral representation, its simpliest integral representation that we know is a  double integral formula, our finite field analogue definition is based on that double integral formula, which naturally leads to more complicated than the definition for $F_{1}$ in \cite{LLM}. Consequently,  the transformations and reduction formulas and  the generating functions for the Appell series $F_{2}$ over finite fields are also  more complicated than those corresponding results in \cite{LLM}.

Here is the outline for the rest of this work. In Section 2 we will prove Theorem \ref{t1}, while in Section 3 several transformation and reduction formulas for $F_{2}(A;B,B';C,C';x,y)$ will be given. Our  last section is devoted to deriving some generating functions for $F_{2}(A;B,B';C,\\
C';x,y).$
\section{Proof of Theorem \ref{t1}}
To carry out our study, we need some auxiliary results which will be used frequently  in this paper.

The results in the  following proposition  follows readily from some properties of Jacobi sums.
\begin{prop}\label{pp2-1} \emph{(See \cite [(2.6), (2.7), (2.8) and (2.12)]{Gr})} If $A,B\in \widehat{\mathbb{F}_{q}},$ then
\begin{align}
  {A\choose B}&={A\choose A\overline{B}},\label{f2}\\
  {A\choose B}&={B\overline{A}\choose B}B(-1)\label{f1},\\
  {A\choose B}&={\overline{B}\choose \overline{A}}AB(-1),\label{f3}\\
{A\choose \varepsilon}&={A\choose A}=-1+(q-1)\delta(A),\label{f4}
\end{align}
where $\delta(\chi)$ is a function on characters given by
\begin{equation*}
\delta(\chi)=\left\{
               \begin{array}{ll}
                 1  & \hbox{if $\chi=\varepsilon$} \\
                 0  & \hbox{otherwise}
               \end{array}.
             \right.
\end{equation*}
\end{prop}
The following result is also very important in the derivation of Theorem \ref{t1}.
\begin{prop}\label{T} For any character $A\in \widehat{\mathbb{F}_{q}}$ and $x,y \in \mathbb{F}_{q},$ we have
\begin{equation*}
  A(1+x+y)=\left\{
             \begin{array}{ll}
               A(x)  & \hbox{if $y=-1$} \\
               \delta(x)\delta(y)+\frac{1}{q-1}\left(\delta(x)\sum_{\chi}{A\choose \chi}\chi(y)+\delta(y)\sum_{\chi}{A\choose \chi}\chi(x)\right)& \\
 ~~~~~~~~~~~~~~~~~~~~~~~~~~~~+\frac{1}{(q-1)^2}\sum_{\chi, \lambda}{A\choose \chi}{A\overline{\chi}\choose \lambda}\chi(x)\lambda(y) & \hbox{if $y\neq-1$}
             \end{array},
           \right.
\end{equation*}
where each sum is over all multiplicative characters of $\mathbb{F}_{q}.$
\end{prop}
\noindent \emph{Proof.}
It is obvious that $A(1+x+y)=A(x)$ when $y=-1.$ We only need to consider the case $y\neq -1.$  When $y\neq -1,$ by the binomial theorem, we have
\begin{align*}
  A(1+x+y)&=A(1+y)A\left(1+\frac{x}{1+y}\right)\\
&=\delta(x)A(1+y)+\frac{1}{q-1}\sum_{\chi}{A\choose \chi}\chi(x)A\overline{\chi}(1+y)\\
&=\delta(x)\delta(y)+\frac{1}{q-1}\left(\delta(x)\sum_{\chi}{A\choose \chi}\chi(y)+\delta(y)\sum_{\chi}{A\choose \chi}\chi(x)\right)\\
&~+\frac{1}{(q-1)^2}\sum_{\chi, \lambda}{A\choose \chi}{A\overline{\chi}\choose \lambda}\chi(x)\lambda(y).
\end{align*}
This completes the proof of Proposition \ref{T}.\qed

Actually, Proposition \ref{T} can be considered as the finite field analogue of the trinomial theorem:
\begin{equation*}
(1+x+y)^{a}=\sum_{i,j\geq 0}{a\choose i}{a-i\choose j}x^i y^j.
\end{equation*}

We now turn to our proof of Theorem \ref{t1}\\
\emph{Proof of Theorem \ref{t1}.} It is clear that $F_{2}(A;B,B';C,C';x,y)=0$ for $y=0.$ We now consider the case $y\neq 0.$
When $y\neq 0,$ if $v=y^{-1},$ then
\begin{equation*}
  \overline{A}(1-ux-vy)=\overline{A}(-ux);
\end{equation*}
if $v\neq y^{-1},$ then from Proposition \ref{T}, we have
\begin{align*}
  \overline{A}(1-ux-vy)&=\delta(ux)\delta(v)+\frac{1}{q-1}\left(\delta(ux)\sum_{\chi}{\overline{A}\choose \chi}\chi(-vy)+\delta(v)\sum_{\chi}{\overline{A}\choose \chi}\chi(-ux)\right) \\
&~+\frac{1}{(q-1)^2}\sum_{\chi, \lambda}{\overline{A}\choose \chi}{\overline{A}\overline{\chi}\choose \lambda}\chi(-ux)\lambda(-vy).
\end{align*}
It is easily seen from the binomial theorem that
\begin{equation*}
  \sum_{\chi, \lambda}{\overline{A}\choose \chi}{\overline{A}\overline{\chi}\choose \lambda}\chi(-ux)\lambda(-1)=\sum_{\chi}{\overline{A}\choose \chi}\chi(-ux)\sum_{\lambda}{\overline{A}\overline{\chi}\choose \lambda}\lambda(-1)=0,
\end{equation*}
which implies that
\begin{align*}
\sum_{u\in \mathbb{F}_{q}}B(u)B'(y^{-1})\overline{B}C(1-u)\overline{B'}C'(1-y^{-1})\sum_{\chi, \lambda}{\overline{A}\choose \chi}{\overline{A}\overline{\chi}\choose \lambda}\chi(-ux)\lambda(-1)=0.
\end{align*}
Then, by \eqref{f1},
\begin{align*}
    &\sum_{u\in \mathbb{F}_{q},v\neq y^{-1}}B(u)B'(v)\overline{B}C(1-u)\overline{B'}C'(1-v)\sum_{\chi, \lambda}{\overline{A}\choose \chi}{\overline{A}\overline{\chi}\choose \lambda}\chi(-ux)\lambda(-vy)\\
&=\sum_{u,v\in \mathbb{F}_{q}}B(u)B'(v)\overline{B}C(1-u)\overline{B'}C'(1-v)\sum_{\chi, \lambda}{\overline{A}\choose \chi}{\overline{A}\overline{\chi}\choose \lambda}\chi(-ux)\lambda(-vy)\\
&~-\sum_{u\in \mathbb{F}_{q}}B(u)B'(y^{-1})\overline{B}C(1-u)\overline{B'}C'(1-y^{-1})\sum_{\chi, \lambda}{\overline{A}\choose \chi}{\overline{A}\overline{\chi}\choose \lambda}\chi(-ux)\lambda(-1)\\
&=\sum_{\chi, \lambda}{A\chi\choose \chi}{A\chi\lambda\choose \lambda}{\overline{C}\overline{\chi}\choose B\overline{C}}{\overline{C'}\overline{\lambda}\choose B'\overline{C'}}\chi(x)\lambda(y).
\end{align*}
Thus, by the fact that $\varepsilon(xy)\delta(ux)B(u)=\delta(v)B'(v)=0,$ \eqref{f2} and \eqref{f1},
\begin{align*}
&F_{2}(A;B,B';C,C';x,y)=\varepsilon(xy)BB'CC'(-1)\left(\sum_{u\in \mathbb{F}_{q},v=y^{-1}}+\sum_{u\in \mathbb{F}_{q},v\neq y^{-1}}\right)\\
&=\overline{A}(-x)\overline{C'}(y)\overline{B'}C'(1-y){\overline{A}B\choose B\overline{C}}+BB'CC'(-1)\sum_{v\neq y^{-1}}B(u)B'(v)\overline{B}C(1-u)\overline{B'}C'(1-v)\\
&~~\cdot  \frac{1}{(q-1)^2}\sum_{\chi, \lambda}{\overline{A}\choose \chi}{\overline{A}\overline{\chi}\choose \lambda}\chi(-ux)\lambda(-vy)\\
&=\frac{1}{(q-1)^2}\sum_{\chi, \lambda}{A\chi\choose \chi}{A\chi\lambda\choose \lambda}{B\chi\choose C\chi}{B'\lambda\choose C'\lambda}\chi(x)\lambda(y)+\overline{A}(-x)\overline{C'}(y)\overline{B'}C'(1-y){\overline{A}B\choose B\overline{C}}.
\end{align*}
In view of the above, we complete the proof of Theorem \ref{t1}.\qed
\section{Reduction and Transformation formulae}
In this section we give some reduction and transformation formulae for $F_{2}(A;B,B';C,C';x,y).$
In order to derive these formulae we need some auxiliary results.
\begin{prop}\label{p2}\emph{(See \cite [Corollary 3.16 and Theorem 3.15]{Gr})} For any $A,B,C,D\in \widehat{\mathbb{F}_{q}}$ and $x \in \mathbb{F}_{q},$ we have
\begin{align}
  {}_{2}F_1 \left(\begin{matrix}
A, \varepsilon \\
C \end{matrix}
\bigg| x \right)&={C\choose A}A(-1)\overline{C}(x)\overline{A}C(1-x)
-C(-1)\varepsilon(x)\label{p2-1}\\
&~~+(q-1)A(-1)\delta(1-x)\delta(\overline{A}C),\notag\\
{}_{2}F_1 \left(\begin{matrix}
A, B \\
A \end{matrix}
\bigg| x \right)&={B\choose A}\varepsilon(x)\overline{B}(1-x)-\overline{A}(-x)\label{p2-33}\\
&~~+(q-1)A(-1)\delta(1-x)\delta(B),\notag\\
{}_{3}F_2 \left(\begin{matrix}
A,B,C\\
A,D \end{matrix}
\bigg| x \right)&={B\choose A}{}_{2}F_1 \left(\begin{matrix}
B,C \\
D \end{matrix}
\bigg| x \right)-\overline{A}(-x){C\overline{A}\choose D\overline{A}}\label{p2-2}\\
&~~+(q-1)A(-1)\overline{D}(x)\overline{C}D(1-x)\delta(B).\notag
\end{align}
\end{prop}

From the definition of $F_{2}(a;b,b';c,c';x,y)$ we know that
\begin{align*}
  F_{2}(a;b,0;c,c';x,y)&={}_{2}F_1 \left(\begin{matrix}
a, b \\
c \end{matrix}
\bigg| x \right),\\
F_{2}(a;0,b';c,c';x,y)&={}_{2}F_1 \left(\begin{matrix}
a, b' \\
c' \end{matrix}
\bigg| y \right).
\end{align*}
We now give a finite field analogue for the above identities.
\begin{thm}\label{t31}Let $A,B,B', C,C'\in \widehat{\mathbb{F}_{q}}$ and $x,~y \in \mathbb{F}_{q}.$  If $y\neq 1,$ then
\begin{align}
  &F_{2}(A;B,\varepsilon;C,C';x,y)\label{t31-1}\\
&=-\varepsilon(y)C'(-1){}_{2}F_1 \left(\begin{matrix}
A, B \\
C \end{matrix}
\bigg| x \right)+\overline{C'}(y)\overline{A}C'(1-y){A\overline{C'}\choose A}{}_{2}F_1 \left(\begin{matrix}
A\overline{C'}, B \\
C \end{matrix}
\bigg| \frac{x}{1-y} \right)\notag\\
&~+(q-1)A(-1)\overline{C}(x)\overline{C'}(y)\overline{B}C^2(1-y)B\overline{C}(1-x-y)\delta(A\overline{C'});\notag
\end{align}
if $x\neq 1,$ then
\begin{align}
  &F_{2}(A;\varepsilon, B';C,C';x,y)\label{t31-2}\\
&=-\varepsilon(x)C(-1){}_{2}F_1 \left(\begin{matrix}
A, B' \\
C' \end{matrix}
\bigg| y \right)+\overline{C}(x)\overline{A}C(1-x){A\overline{C}\choose A}{}_{2}F_1 \left(\begin{matrix}
A\overline{C}, B' \\
C' \end{matrix}
\bigg| \frac{y}{1-x} \right)\notag\\
&~+(q-1)A(-1)\overline{C'}(y)\overline{C}(x)\overline{B'}C'^2(1-x)B'\overline{C'}(1-x-y)\delta(A\overline{C}).\notag
\end{align}
\end{thm}
\noindent{\it Proof.} We first prove \eqref{t31-1}. It follows from  \eqref{p2-1} that
\begin{align*}
  \sum_{\lambda}{A\chi\lambda\choose \lambda}{\lambda\choose C'\lambda}\lambda(y)&=(q-1){}_{2}F_1 \left(\begin{matrix}
A\chi, \varepsilon \\
C' \end{matrix}
\bigg| y \right)\\
&=(q-1){C'\choose A\chi}A\chi(-1)\overline{C'}(y)\overline{A}\overline{\chi}C'(1-y)-(q-1)C'(-1)\varepsilon(y).
\end{align*}
Then, using the above identity in Theorem \ref{t1}, by  \eqref{f2}--\eqref{f3}, \eqref{e1-1} and \eqref{p2-2}, and canceling some terms, we get
\begin{align*}
  &F_{2}(A;B,\varepsilon;C,C';x,y)\\
&=\frac{1}{(q-1)^2}\sum_{\chi}{A\chi\choose \chi}{B\chi\choose C\chi}\chi(x)\sum_{\lambda}{A\chi\lambda\choose \lambda}{\lambda\choose C'\lambda}\lambda(y)+\overline{A}(-x)\overline{C'}(y)C'(1-y){\overline{A}B\choose B\overline{C}}\\
&=\overline{C'}(y)\overline{A}C'(1-y)
{}_{3}F_2 \left(\begin{matrix}
A, A\overline{C'},B \\
 A,C \end{matrix}
\bigg| \frac{x}{1-y} \right)
-\varepsilon(y)C'(-1){}_{2}F_1 \left(\begin{matrix}
A, B \\
C \end{matrix}
\bigg| x \right)\\
&~+\overline{A}(-x)\overline{C'}(y)C'(1-y){\overline{A}B\choose B\overline{C}}\\
&=-\varepsilon(y)C'(-1){}_{2}F_1 \left(\begin{matrix}
A, B \\
C \end{matrix}
\bigg| x \right)+\overline{C'}(y)\overline{A}C'(1-y){A\overline{C'}\choose A}{}_{2}F_1 \left(\begin{matrix}
A\overline{C'}, B \\
C \end{matrix}
\bigg| \frac{x}{1-y} \right)\\
&~+(q-1)A(-1)\overline{C}(x)\overline{C'}(y)\overline{A}\overline{B}C^2C'(1-y)B\overline{C}(1-x-y)\delta(A\overline{C'}).
\end{align*}
This proves \eqref{t31-1}.

Identity \eqref{t31-2} follows from \eqref{t31-1} and \eqref{p11-1}. This completes the proof of Theorem \ref{t31}. \qed

In \cite [\S 9.5, (3)]{B}, Bailey gave the following reduction formula:
\begin{equation*}
F_{2}(a;b,b';b,c';x,y)=(1-x)^{-a}{}_{2}F_1 \left(\begin{matrix}
a, b' \\
c' \end{matrix}
\bigg| \frac{y}{1-x} \right).
\end{equation*}
We also deduce the finite field analogue for the above formula.
\begin{thm}\label{t32}For any $A,B,B', C'\in \widehat{\mathbb{F}_{q}}$ and $x\in \mathbb{F}_{q}\backslash \{1\},~y \in \mathbb{F}_{q},$ we have
\begin{align*}
  &F_{2}(A;B,B';B,C';x,y)\\
&=-\varepsilon(x)\overline{A}(1-x){}_{2}F_1 \left(\begin{matrix}
A, B' \\
C' \end{matrix}
\bigg| \frac{y}{1-x} \right)+\overline{B}(x){A\overline{B}\choose \overline{B}}{}_{2}F_1 \left(\begin{matrix}
A\overline{B},B' \\
C' \end{matrix}
\bigg|y \right)\\
&~+(q-1)\delta(A\overline{B})\overline{A}(-x)\overline{C'}(y)\overline{B'}C'(1-y).
\end{align*}
\end{thm}
\noindent{\it Proof.} We know from \cite [(2.11)]{Gr} that for any $A,~B\in \widehat{\mathbb{F}_{q}}$ and $x\in \mathbb{F}_{q},$ we have
\begin{equation}\label{t32-1}
\sum_{\chi}{A\chi\choose B\chi}\chi(x)=(q-1)\overline{B}(x)\overline{A}B(1-x).
\end{equation}
It is easily seen from \eqref{p2-33} that
\begin{align*}
  \sum_{\chi}{A\chi\choose \chi}{A\lambda\chi\choose A\chi}\chi(x)&=(q-1){}_{2}F_1 \left(\begin{matrix}
A, A\lambda \\
A \end{matrix}
\bigg| x \right)\\
&=(q-1){A\lambda\choose A}\varepsilon(x)\overline{A}\overline{\lambda}(1-x)-(q-1)\overline{A}(-x)\label{p2-33}.
\end{align*}
Then, by \eqref{e1-1} and \eqref{t32-1},
\begin{equation}
\begin{aligned}
&\sum_{\chi, \lambda}{A\chi\choose \chi}{A\chi\lambda\choose A\chi}{B'\lambda\choose C'\lambda}\chi(x)\lambda(y)\label{e}\\
&=\sum_{\lambda}{B'\lambda\choose C'\lambda}\lambda(y)\sum_{\chi}{A\chi\choose \chi}{A\lambda\chi\choose A\chi}\chi(x)\\
&=(q-1)^2\varepsilon(x)\overline{A}(1-x){}_{2}F_1 \left(\begin{matrix}
A, B' \\
C' \end{matrix}
\bigg| \frac{y}{1-x} \right)-(q-1)^2\overline{A}(-x)\overline{C'}(y)\overline{B'}C'(1-y).
\end{aligned}
\end{equation}
Thus, using \eqref{f2} and \eqref{f4} and cancelling some terms to give
\begin{align*}
  &F_{2}(A;B,B';B,C';x,y)\\
&=\frac{1}{(q-1)^2}\sum_{\chi, \lambda}{A\chi\choose \chi}{A\chi\lambda\choose A\chi}\left(-1+(q-1)\delta(B\chi)\right){B'\lambda\choose C'\lambda}\chi(x)\lambda(y)\\
&~+\overline{A}(-x)\overline{C'}(y)\overline{B'}C'(1-y)\left(-1+(q-1)\delta(A\overline{B})\right)\\
&=-\varepsilon(x)\overline{A}(1-x){}_{2}F_1 \left(\begin{matrix}
A, B' \\
C' \end{matrix}
\bigg| \frac{y}{1-x} \right)+\overline{B}(x){A\overline{B}\choose \overline{B}}{}_{2}F_1 \left(\begin{matrix}
A\overline{B},B' \\
C' \end{matrix}
\bigg|y \right)\\
&~+(q-1)\delta(A\overline{B})\overline{A}(-x)\overline{C'}(y)\overline{B'}C'(1-y).
\end{align*}
This concludes the proof of Theorem \ref{t32}. \qed

From Theorem \ref{t32} and \eqref{p11-1} we can easily derive the following identity which is the finite field analogue for the formula
\begin{equation*}
F_{2}(a;b,b';c,b';x,y)=(1-y)^{-a}{}_{2}F_1 \left(\begin{matrix}
a, b \\
c \end{matrix}
\bigg| \frac{x}{1-y} \right).
\end{equation*}
\begin{thm}For any $A,B,B', C\in \widehat{\mathbb{F}_{q}}$ and $x\in \mathbb{F}_{q},~y \in \mathbb{F}_{q}\backslash \{1\},$ we have
\begin{align*}
  &F_{2}(A;B,B';C,B';x,y)\\
&=-\varepsilon(y)\overline{A}(1-y){}_{2}F_1 \left(\begin{matrix}
A, B \\
C \end{matrix}
\bigg| \frac{x}{1-y} \right)+\overline{A}(-y)\overline{C}(x)\overline{B}C(1-x)\\
&~+\overline{A}(-y)\overline{C}(x)\overline{B}C(1-x)\left(-1+(q-1)\delta(A\overline{B'})\right)\\
&~+\overline{B'}(y){A\overline{B'}\choose \overline{B'}}{}_{2}F_1 \left(\begin{matrix}
A\overline{B'},B\\
C \end{matrix}
\bigg|x \right).
\end{align*}
\end{thm}



From the definition of $F_{2}(A;B,B';C,C';x,y),$ we can easily deduce the following transformation formulae for $F_{2}(A;B,B';C,C';x,y).$
\begin{thm}\label{t41}For any $A,B,B',C,C'\in \widehat{\mathbb{F}_{q}} $ and $x,y\in \mathbb{F}_{q},$ we have
\begin{align}
  &F_{2}(A;B,B';C,C';x,y) \label{t411}\\
&=C(-1)\overline{A}(1-x)F_{2}\left(A;\overline{B}C,B';C,C';-\frac{x}{1-x},\frac{y}{1-x}\right)\notag\\
&=C'(-1)\overline{A}(1-y)F_{2}\left(A;B,\overline{B'}C';C,C';\frac{x}{1-y},-\frac{y}{1-y}\right)\notag\\
&=CC'(-1)\overline{A}(1-x-y)F_{2}\left(A;\overline{B}C,\overline{B'}C';C,C';-\frac{x}{1-x-y},-\frac{y}{1-x-y}\right),\notag
\end{align}
where the first identity holds for $x\neq 1,$ the second identity holds for $y\neq 1$ and the third identity holds for $x+y\neq 1.$
\end{thm}
\begin{proof}Using the definition of $F_{2}(A;B,B';C,C';x,y)$ and then  making the substitutions (1) $u=1-u',~v=v',$ (2) $u=u',~v=1-v',$ (3) $u=1-u',~v=1-v'$ at the left side of \eqref{t411} we can easily obtain these transformation formulae.
\end{proof}
It is easily seen that these transformation formulae in \eqref{t411} can be regarded as the finite field analogue of \cite [\S 9.4, (6)--(8)]{B}.

From Theorem \ref{t41} and\eqref{p11-2},  we can easily deduce the  following result  involving $F_{2}(A;B,B';
\\C,C'; x,y)$ for the case $x+y=1.$
\begin{thm}For any $A,B,B',C,C'\in \widehat{\mathbb{F}_{q}} $ and $x\in \mathbb{F}_{q}\backslash \{1\},$ we have
\begin{equation*}
  F_{2}(A;B,B';C,C';x,1-x)=CB'C'(-1)\overline{A}(1-x) {}_{3}F_2 \left(\begin{matrix}
A,\overline{B}C,A\overline{C'}\\
C,AB'\overline{C'} \end{matrix}
\bigg| -\frac{x}{1-x} \right).
\end{equation*}
\end{thm}
\section{Generating functions}
In this section, we establish some generating functions for $F_{2}(A;B,B';C,C';x,y).$

We first state a result of Greene  in our notations.
\begin{prop}\label{pp4-1}\emph{(See \cite [(2.15)]{Gr})} For any $A,B,C\in \widehat{\mathbb{F}_{q}},$  we have
\begin{align*}
  {A\choose B}{C\choose A}={C\choose B}{C\overline{B} \choose A\overline{B}}-(q-1)B(-1)\delta(A)+(q-1)AB(-1)\delta(B\overline{C}).
\end{align*}
\end{prop}
The following theorem involves a generating function for $F_{2}(A;B,B';C,C';x,y).$
\begin{thm}\label{t4-2}For any $A,B,B',C,C'\in \widehat{\mathbb{F}_{q}} $ and $x\in \mathbb{F}^*_{q},~y\in \mathbb{F}_{q},~t\in \mathbb{F}^*_{q}\backslash\{1\},$ we have
\begin{align*}
  &\frac{1}{q-1}\sum_{\theta}{A\theta\choose\theta}F_{2}(A\theta;B,B';C,C';x,y)\theta(t)\\
&=\overline{A}(1-t)F_{2}\left(A;B,B';C,C';\frac{x}{1-t}, \frac{y}{1-t}\right)\\
&~-\overline{A}(-t)\overline{C'}(y)\overline{B'}C'(1-y){}_{2}F_1 \left(\begin{matrix}
A, B \\
C \end{matrix}
\bigg| -\frac{x}{t} \right)\\
&~-\overline{A}(-t)\left({B\choose C}{B'\choose C'}\varepsilon(y)-F_{2}(\varepsilon;B,B';C,C';x,y)\right)\\
&~+BC(-1)\overline{A}(-x)\overline{C'}(y)\overline{B'}C'(1-y){}_{2}F_1 \left(\begin{matrix}
A, A\overline{C} \\
A\overline{B} \end{matrix}
\bigg| -\frac{t}{x} \right).
\end{align*}
\end{thm}
\noindent{\it Proof.} It is easily seen from \eqref{p2-33} that
\begin{align*}
  {}_{2}F_1 \left(\begin{matrix}
A\chi, A\chi\lambda \\
 A\chi \end{matrix}
\bigg|t\right)={A\chi\lambda\choose A\chi}\overline{A}\overline{\chi}\overline{\lambda}(1-t)-\overline{A}\overline{\chi}(-t).
\end{align*}
Then from \eqref{p2-2} we know that
\begin{align*}
  &{}_{3}F_2 \left(\begin{matrix}
A, A\chi, A\chi\lambda \\
A, A\chi \end{matrix}
\bigg|t\right)\\
&={A\chi\choose A}{}_{2}F_1 \left(\begin{matrix}
A\chi, A\chi\lambda \\
 A\chi \end{matrix}
\bigg|t\right)-\overline{A}(-t){\chi\lambda\choose \chi}+(q-1)A(-1)\overline{A}\overline{\chi}(t)\overline{\lambda}(1-t)\delta(A\chi)\\
&={A\chi\choose \chi}{A\chi\lambda\choose A\chi}\overline{A}\overline{\chi}\overline{\lambda}(1-t)-{A\chi\choose \chi}\overline{A}\overline{\chi}(-t)-\overline{A}(-t){\chi\lambda\choose \chi}\\
&~+(q-1)A(-1)\overline{A}\overline{\chi}(t)\overline{\lambda}(1-t)\delta(A\chi).
\end{align*}
Thus
\begin{align}
  &\sum_{\theta,\chi,\lambda}{A\theta\choose \theta}{A\chi\theta\choose A\theta}{A\chi\lambda\theta\choose A\chi\theta}{B\chi \choose C\chi}{B'\lambda \choose C'\lambda}\chi(x)\lambda(y)\theta(t)\label{4-1}\\
&=(q-1)\sum_{\chi,\lambda}{B\chi \choose C\chi}{B'\lambda \choose C'\lambda}\chi(x)\lambda(y){}_{3}F_2 \left(\begin{matrix}
A, A\chi, A\chi\lambda \\
A, A\chi \end{matrix}
\bigg|t\right)\notag\\
&=(q-1)\overline{A}(1-t)\sum_{\chi,\lambda}{A\chi\choose \chi}{A\chi\lambda\choose A\chi}{B\chi \choose C\chi}{B'\lambda \choose C'\lambda}\chi\left(\frac{x}{1-t}\right)\lambda\left(\frac{y}{1-t}\right)\notag\\
&~-(q-1)\overline{A}(-t)\sum_{\chi,\lambda}{A\chi\choose \chi}{B\chi \choose C\chi}{B'\lambda \choose C'\lambda}\chi\left(-\frac{x}{t}\right)\lambda(y)\notag\\
&~-(q-1)\overline{A}(-t)\sum_{\chi,\lambda}{\chi\lambda\choose \chi}{B\chi \choose C\chi}{B'\lambda \choose C'\lambda}\chi(x)\lambda(y)\notag
\end{align}
\begin{align*}
&~+(q-1)^2\overline{A}(-x){\overline{A}B\choose \overline{A}C}\sum_{\lambda}{B'\lambda\choose C'\lambda}\lambda\left(\frac{y}{1-t}\right).
\end{align*}
From Theorem \ref{t1} we see that
\begin{align}
  &\sum_{\chi,\lambda}{A\chi\choose \chi}{A\chi\lambda\choose A\chi}{B\chi \choose C\chi}{B'\lambda \choose C'\lambda}\chi\left(\frac{x}{1-t}\right)\lambda\left(\frac{y}{1-t}\right)\label{4-2}\\
&=(q-1)^2F_{2}\left(A;B,B';C,C';\frac{x}{1-t}, \frac{y}{1-t}\right)\notag\\
&~-(q-1)^2\overline{A}(-x)\overline{C'}(y)\overline{B'}C'(1-t-y)AB'(1-t){\overline{A}B\choose B\overline{C}}.\notag
\end{align}
By \eqref{f1} and \eqref{t32-1},
\begin{align}
 &\sum_{\chi,\lambda}{A\chi\choose \chi}{B\chi \choose C\chi}{B'\lambda \choose C'\lambda}\chi\left(-\frac{x}{t}\right)\lambda(y)\label{4-3}\\
&=\sum_{\chi}{A\chi\choose \chi}{B\chi \choose C\chi}\chi\left(-\frac{x}{t}\right)\sum_{\lambda}{B'\lambda \choose C'\lambda}\lambda(y)\notag\\
&=(q-1)^2\overline{C'}(y)\overline{B'}C'(1-y){}_{2}F_1 \left(\begin{matrix}
A, B \\
C \end{matrix}
\bigg| -\frac{x}{t} \right). \notag
\end{align}
It can be deduced from \eqref{f4} and \eqref{t32-1} that
\begin{align*}
  \sum_{\lambda}{\lambda\choose \lambda}{B'\lambda\choose C'\lambda}\lambda(y)&=-\sum_{\lambda}{B'\lambda\choose C'\lambda}\lambda(y)+(q-1){B'\choose C'}\varepsilon(y)\\
&=-(q-1)\overline{C'}(y)\overline{B'}C'(1-y)+(q-1){B'\choose C'}\varepsilon(y).
\end{align*}
This combines \eqref{f4} to give
\begin{align*}
  &\sum_{\chi,\lambda}{\chi\choose \chi}{\chi\lambda\choose \chi}{B\chi \choose C\chi}{B'\lambda \choose C'\lambda}\chi(x)\lambda(y)\\
&=(q-1){B\choose C}\sum_{\lambda}{\lambda\choose \lambda}{B'\lambda\choose C'\lambda}\lambda(y)-\sum_{\chi,\lambda}{\chi\lambda\choose \chi}{B\chi \choose C\chi}{B'\lambda \choose C'\lambda}\chi(x)\lambda(y)\\
&=(q-1)^2 {B\choose C}{B'\choose C'}\varepsilon(y)-(q-1)^2 \overline{C'}(y)\overline{B'}C'(1-y){B\choose C}\\
&~-\sum_{\chi,\lambda}{\chi\lambda\choose \chi}{B\chi \choose C\chi}{B'\lambda \choose C'\lambda}\chi(x)\lambda(y).
\end{align*}
From Theorem \ref{t1} and \eqref{f2} we have
\begin{align*}
 \sum_{\chi,\lambda}{\chi\choose \chi}{\chi\lambda\choose \chi}{B\chi \choose C\chi}{B'\lambda \choose C'\lambda}\chi(x)\lambda(y)&=(q-1)^2F_{2}(\varepsilon;B,B';C,C';x,y)\\
&~-(q-1)^2\overline{C'}(y)\overline{B'}C'(1-y){B\choose C}.
\end{align*}
So we deduce from the above two identities that
\begin{align}
  \sum_{\chi,\lambda}{\chi\lambda\choose \chi}{B\chi \choose C\chi}{B'\lambda \choose C'\lambda}\chi(x)\lambda(y)&=(q-1)^2 {B\choose C}{B'\choose C'}\varepsilon(y)\label{4-5}\\
&~-(q-1)^2F_{2}(\varepsilon;B,B';C,C';x,y).\notag
\end{align}
By Theorem \ref{t1} and \eqref{f2}--\eqref{f3},
\begin{align*}
  &\sum_{\theta}{A\theta\choose\theta}F_{2}(A\theta;B,B';C,C';x,y)\theta(t)\\
&=\frac{1}{(q-1)^2}\sum_{\theta,\chi,\lambda}{A\theta\choose \theta}{A\chi\theta\choose A\theta}{A\chi\lambda\theta\choose A\chi\theta}{B\chi \choose C\chi}{B'\lambda \choose C'\lambda}\chi(x)\lambda(y)\theta(t)\\
&~+BC(-1)\overline{A}(-x)\overline{C'}(y)\overline{B'}C'(1-y)\sum_{\theta}{A\theta\choose \theta}{A\overline{C}\theta\choose A\overline{B}\theta}\theta(-t/x).
\end{align*}
Using \eqref{4-2}--\eqref{4-5} and \eqref{t32-1} in \eqref{4-1}, substituting \eqref{4-1} in the above identity and cancelling two terms, we obtain  the result
to ends the proof of Theorem \ref{t4-2}.\qed

Theorem \ref{t4-2} is actually the finite field analogue for \cite [(2.2)]{CA}.

We also establish two other generating functions for $F_{2}(A;B,B';C,C';x,y).$
\begin{thm}\label{t4-1}Let $A,B,B',C,C'\in \widehat{\mathbb{F}_{q}} $ and $x,y\in \mathbb{F}_{q},~t\in \mathbb{F}_{q}\backslash \{1\}.$  If $x\neq 1,$ then
\begin{align}
 &\frac{1}{q-1}\sum_{\theta}{B\overline{C}\theta\choose\theta}F_{2}(A;B\theta,B';C,C';x,y)\theta(t)\label{t41-1}\\
&=\varepsilon(t)\overline{B}(1-t)F_{2}\left(A;B,B';C,C';\frac{x}{1-t},y\right) \notag\\
&~-\overline{B}C(-t)\varepsilon(x)\overline{A}(1-x){}_{2}F_1 \left(\begin{matrix}
A, B' \\
C' \end{matrix}
\bigg| \frac{y}{1-x} \right);\notag
\end{align}
if $y\neq 1,$ then
\begin{align}
&\frac{1}{q-1}\sum_{\theta}{B'\overline{C'}\theta\choose\theta}F_{2}(A;B,B'\theta;C,C';x,y)\theta(t)\label{t41-2}\\
&=\varepsilon(t)\overline{B'}(1-t)F_{2}\left(A;B,B';C,C';x,\frac{y}{1-t}\right) \notag\\
&~-\overline{B'}C'(-t)\varepsilon(y)\overline{A}(1-y){}_{2}F_1 \left(\begin{matrix}
A, B \\
C \end{matrix}
\bigg| \frac{x}{1-y} \right).\notag
\end{align}
\end{thm}
\noindent{\it Proof.} We first prove \eqref{t41-1}.
It is easy to know from Theorem \ref{t1} that
\begin{align}
 &\sum_{\theta}{B\overline{C}\theta\choose\theta}F_{2}(A;B\theta,B';C,C';x,y)\theta(t)\label{5-2}\\
&=\frac{1}{(q-1)^2}\sum_{\theta, \chi, \lambda}{B\overline{C}\theta\choose\theta}{B\theta\chi\choose B\overline{C}\theta}{A\chi\choose \chi}{A\chi\lambda\choose \lambda}{B'\lambda\choose C'\lambda}\chi(x)\lambda(y)\theta(t)\notag\\
&~+\overline{A}(-x)\overline{C'}(y)\overline{B'}C'(1-y)\sum_{\theta}{B\overline{C}\theta\choose\theta}{\overline{A}B\theta\choose B\overline{C}\theta}\theta(t).\notag
\end{align}
It can be seen from  Proposition \ref{pp4-1} that
\begin{equation*}
{B\overline{C}\theta\choose\theta}{B\theta\chi\choose B\overline{C}\theta}={B\theta\chi\choose \theta}{B\chi\choose B\overline{C}}-(q-1)\theta(-1)\delta(B\overline{C}\theta)+(q-1)BC(-1)\delta(\overline{B}\overline{\chi}).
\end{equation*}
Then
\begin{align}
  &\sum_{\theta, \chi, \lambda}{B\overline{C}\theta\choose\theta}{B\theta\chi\choose B\overline{C}\theta}{A\chi\choose \chi}{A\chi\lambda\choose \lambda}{B'\lambda\choose C'\lambda}\chi(x)\lambda(y)\theta(t)\label{5-1}\\
&=\sum_{\theta, \chi, \lambda}{B\theta\chi\choose \theta}{B\chi\choose B\overline{C}}{A\chi\choose \chi}{A\chi\lambda\choose \lambda}{B'\lambda\choose C'\lambda}\chi(x)\lambda(y)\theta(t)\notag\\
&~-(q-1)\overline{B}C(-t)\sum_{\chi, \lambda}{A\chi\choose \chi}{A\chi\lambda\choose \lambda}{B'\lambda\choose C'\lambda}\chi(x)\lambda(y).\notag
\end{align}
It follows from \eqref{f2}, \eqref{t32-1} and Theorem \ref{t1} that
\begin{align}
  &\sum_{\theta, \chi, \lambda}{B\theta\chi\choose \theta}{B\chi\choose B\overline{C}}{A\chi\choose \chi}{A\chi\lambda\choose \lambda}{B'\lambda\choose C'\lambda}\chi(x)\lambda(y)\theta(t)\label{5-3}\\
&=\sum_{ \chi, \lambda}{A\chi\choose \chi}{A\chi\lambda\choose \lambda}{B\chi\choose C\chi}{B'\lambda\choose C'\lambda}\chi(x)\lambda(y)\sum_{\theta}{B\chi\theta\choose \theta}\theta(t)\notag\\
&=(q-1)\varepsilon(t)\overline{B}(1-t)\sum_{ \chi, \lambda}{A\chi\choose \chi}{A\chi\lambda\choose \lambda}{B\chi\choose C\chi}{B'\lambda\choose C'\lambda}\chi\left(\frac{x}{1-t}\right)\lambda(y)\notag\\
&=(q-1)^3\varepsilon(t)\overline{B}(1-t)F_{2}\left(A;B,B';C,C';\frac{x}{1-t},y\right)\notag\\
&~-(q-1)^3\varepsilon(t)A\overline{B}(1-t)\overline{A}(-x)\overline{C'}(y)\overline{B'}C'(1-y){\overline{A}B\choose B\overline{C}}.\notag
\end{align}
It is easily known from \eqref{p2-33} that
\begin{align}
  \sum_{\theta}{B\overline{C}\theta\choose\theta}{\overline{A}B\theta\choose B\overline{C}\theta}\theta(t)&=(q-1){}_{2}F_1 \left(\begin{matrix}
B\overline{C}, \overline{A}B \\
B\overline{C} \end{matrix}
\bigg| t \right)\label{q41-2}\\
&=(q-1){\overline{A}B\choose B\overline{C}}\varepsilon(t)A\overline{B}(1-t)-(q-1)\overline{B}C(-t).\notag
\end{align}
Using \eqref{5-3}  and \eqref{e} in \eqref{5-1}, combining \eqref{5-2}, \eqref{5-1} and \eqref{q41-2} and canceling some terms, we get
\begin{align*}
  &\sum_{\theta}{B\overline{C}\theta\choose\theta}F_{2}(A;B\theta,B';C,C';x,y)\theta(t)\\
&=(q-1)\varepsilon(t)\overline{B}(1-t)F_{2}\left(A;B,B';C,C';\frac{x}{1-t},y\right) \notag\\
&~-(q-1)\overline{B}C(-t)\varepsilon(x)\overline{A}(1-x){}_{2}F_1 \left(\begin{matrix}
A, B' \\
C' \end{matrix}
\bigg| \frac{y}{1-x} \right),\notag
\end{align*}
which proves \eqref{t41-1}.

Identity \eqref{t41-2} follows easily from \eqref{t41-1} and \eqref{p11-1}. This finishes the proof of Theorem \ref{t4-1}. \qed
\section*{Acknowledgements}
The first author was partially  supported by the Initial Foundation for Scientific Research  of Northwest A\&F University (No. 2452015321).  The second author was partially  supported by the Natural Science Foundation of China (grant No.11571114). The third author was partially  supported by Natural Science Foundation of China (grant No. 11371294).


\begin{thebibliography}{99}
\small \setlength{\itemsep}{0.1mm}
\bibitem{A}
P. Agarwal, On new unified integrals involving Appell series, Advances in Mechanical Engineering and its Applications, 2 (2012), no. 1, 115--120.
\bibitem{B}
W.N. Bailey,  Generalized Hypergeometric Series. Cambridge: Cambridge University Press;  1935.
\bibitem{BEW}
B.C.  Berndt,  R.J. Evans and  K.S. Williams,  Gauss and Jacobi sums. Canadian Mathematical Society Series of Monographs and Advanced Texts. A Wiley--Interscience Publication. John Wiley \& Sons, Inc., New York, 1998.
\bibitem{CA}
J. Choi and P. Agarwal, Certain generating functions involving Appell series, Far
East Journal of Mathematical Sciences. 84 (2014), no. 1, 25--32.
\bibitem{EG}
R.  Evans  and  J. Greene,  Evaluations of hypergeometric functions over finite fields. Hiroshima Math. J.  39  (2009),  no. 2, 217--235.
\bibitem{FL}
J. Fuselier, L. Long, R. Ramakrishna, H. Swisher and F.-T. Tu, Hypergeometric Functions over Finite Fields, arXiv: 1510.02575.
\bibitem{Gr}
J. Greene, Hypergeometric functions over finite fields, Trans. Amer. Math. Soc. 301
(1987), no. 1, 77--101.
\bibitem{IR}
K.  Ireland and M. Rosen,  A classical introduction to modern number theory. Second edition. Graduate Texts in Mathematics, 84. Springer--Verlag, New York, 1990.
\bibitem{LLM}
L. Li, X. Li and R. Mao, Some new formulas for Appell series over finite fields, arXiv:1701.02674.
\bibitem{M}
D. McCarthy, Transformations of well-poised hypergeometric functions over finite fields. Finite Fields Appl.  18  (2012),  no. 6, 1133--1147.
\bibitem{S}
M.J.  Schlosser, Multiple hypergeometric series: Appell series and beyond. In Computer Algebra
in Quantum Field Theory, pages 305--324. Springer, 2013.
\end{thebibliography}
\end{document}